\documentclass{article}
\usepackage[cp1250]{inputenc}
\usepackage{lmodern}

\usepackage[english]{babel}
\usepackage{polski}
\usepackage{amsmath,amssymb,array,amsthm}

\usepackage{enumitem}

\newtheorem{thm}{Theorem}[section]
\newtheorem{definition}[thm]{Definition}
\newtheorem{lem}[thm]{Lemma}

\newtheorem{remark}[thm]{Remark}
\newtheorem{prop}[thm]{Proposition}
\newtheorem{cor}[thm]{Corollary}

\def\SS{\mathsf{S}}
\def\r{\mathfrak{R}}

\def\wto{\looparrowright}
\def\wast{\ast^\mathcal{E}}
\def\wi{\vartriangle}
\def\m{\mathfrak{M}}
\def\n{\mathfrak{N}}
\def\M{\mathsf{M}}
\def\R{\mathsf{R}}
\def\rthmin{\r^{\tthm}_{\mathsf{min}}}
\def\rthmax{\r^{\tthm}_{\mathsf{max}}}
\def\a{\varphi}
\def\b{\psi}
\def\c{\chi}
\def\for{\mathsf{FOR}}
\def\Sthm{\SS^{\tthm}}
\def\eqm{|\m|_\approx}
\def\model{\langle v,\r\rangle}
\def\nodel{\langle v',\r'\rangle}
\def\splus{\Sigma^+}
\def\th{\mathsf{Th}}
\def\tthm{\th(\m)}
\def\thn{\th(\n)}
\def\vthm{v^{\tthm}}
\def\cpl{\mathsf{CPL}}
\def\splus{\Sigma^+}
\def\Ssplus{\SS^{\splus}}
\def\rsplusmin{\r^{\splus}_{\mathsf{min}}}
\def\rsplusmax{\r^{\splus}_{\mathsf{max}}}

\def\ca2{C^\a_2}
\def\om{\Omega^\m}
\def\sm{\SS^\m}
\def\K{\mathsf{K}}

\def\rF{\mathcal{F}}

\def\vsplus{v^{\splus}}
\def\at{\mathsf{At}}
\def\subst{\mathsf{Subst}}
\def\V{\mathsf{Var}}

\title{Epstein Semantics: Characterization, Interpolation, Undefinability, and (In)Completeness}
\author{Krzysztof A. Krawczyk}
\date{}
\begin{document}
\maketitle
\begin{center}
Abstract
\end{center}
This paper is a mathematical investigation on Epstein semantics. 
One of the main tools of the present paper is the model-theoretic $\SS$-set construction introduced in \cite{ja}. We use it to prove several results: 1) that each Epstein model has uncountably many equivalent Epstein models, 2) that the logic of generalised Epstein models is the $\SS$-set invariant fragment of $\cpl$ (analogon of the celebrated van Benthem characterization theorem for modal logic), 3) that several sets of Epstein models are undefinable, 4) that logics of undefinable sets of relations can be finitely axiomatised. We also use other techniques to prove 
 5) that there is uncountably many Epstein-incomplete logics and that 6) the logic of generalised Epstein models has the interpolation property.
\section{Introduction}
In a series of publications \cite{epstein79, epstein87, epstein90, epstein05}, Richard Epstein has introduced a conceptually simple semantics for non-classical logics. His models are two-element tuples containing standard boolean valuation of propositional variables and a binary relation defined on the set of formulas. Epstein's motivations were mainly philosophical: he wanted to formalise the content-relationship between propositions. However, he extended his research of the so-called relatedness and dependence logics to purely mathematical form. Also, he argued that his semantics can serve as adequate interpretations for other non-classical logics, and can be used as a tool for analysing other, different from the content-relationship, philosophical problems. This perspective has been known as the Epstein programme, see \cite{krajewski82, krajewski86, krajewski91}. Epstein's ideas inspired other researchers in the late seventies \cite{waltona,waltonb}, and studies on Epstein semantics and related topics have been carried out up till nineties \cite{iseminger, carnielli, del, Paoli93, Paoli96}. More recently, Epstein ideas have been picked up in \cite{jarkacz}. Despite some technical deficiencies\footnote{Statement marked as fact has been proven to be false, see \cite[Theorem 3.6 and a comment below]{ja}}, the paper contains a two-folded generalisation of Epstein ideas: relations of the models do not have to meet any conditions, and additional Epstein-style connectives can be introduced. These simple generalisations open a new perspective on the Epstein programme: it can be extended and carried out with greater impetus.

The current work, being a continuation of \cite{ja}, shares a common perspective and ambition with the cited paper: to approach Epstein semantics from a purely technical perspective, and thus, to establish it as a branch of mathematical logic. In the current paper we investigate, exploit and modify the $\SS$-set construction introduced in \cite{ja} for generalised Epstein semantics. We prove that the cardinality of an $\SS$-set is always uncountable. Also, by means of a translation into $\cpl$, we prove an analogon of the well-known van Benthem characterization theorem for modal logic, which will state that generalised Epstein logic is an $\SS$-set invariant fragment of $\cpl$. Later on, we prove completeness theorems for undefinable sets of Epstein models and relations. Both undefinability of a given set of relations/models, as well as completeness theorems are proven by means of an $\SS$--set construction. In the latter case, the construction is modified in a way which enables us to obtain the set of all equivalent models from purely syntactical entities. What is interesting, undefinable sets of structures turn out to be finitely axiomatisable (we do not need to add special rules -- unlike in the case of modal logics). Axiomatisation is achieved by means of special implicational formulas which have the form $...\to(\a\wast\b)\lor\neg(\a\ast\b)$, where $\wast$ stands for an Epstein-style connective. Also, we investigate Epstein-incompleteness - a phenomenon analogous to Kripke-incompleteness in modal logic. We show that the number of Epstein-incomplete logics (the ones that do not have adequate sets of Epstein relations) is $2^{\aleph_0}$. We close the paper with the proof of the interpolation theorem for $\mathcal{F}$. 
\section{Preliminaries}
We shall stick to the language introduced in \cite{jarkacz}. Let $\Phi=\{p_0,p_1,p_2,...\}$ be the set of propositional letters. We have one unary connective: negation $\neg$ and binary connectives: $\lor,\wedge,\to,\leftrightarrow,\wi,\wto$. The set of formulas $\for$ is built in a standard inductive way. We will use $\top$ and $\bot$ as abbreviations for $p_0\lor\neg p_0$ and $\neg(p_0\lor \neg p_0)$ respectively. A substitution is any function $\sigma: \for\longrightarrow\for$ which is an endomorphism of the absolutely free formula algebra with domain $\for$. The set of all substitutions will be denoted by $\subst$.

An Epstein model is an ordered pair: $\m=\model$, where $v:\Phi\longrightarrow\{0,1\}$ is a standard valuation and $\r\subseteq\for\times\for$ is a binary relation defined on the set of formulas. By $\M$ we will denote the set of all models\footnote{We can use the word `set' without committing to a `class' since $\M=\{\model: v\in 2^\Phi, \r\subseteq\for^2\}$}. The set of all relations is $\mathcal{P}(\for^2)$. Let $\a\in\for$, $\m=\model$. We say that $\a$ is true in $\m$, symbolically $\m\vDash\a$ iff: 
\begin{itemize}\itemsep=0pt
\item $v(\a)=1$ for $\a\in\Phi$,
\item $\m\nvDash\b$ for $\a=\neg\b$,
\item $\m\vDash\b$ and $\m\vDash\c$ for $\a=\b\wedge\c$,
\item $\m\vDash\b$ or $\m\vDash\c$ for $\a=\b\lor\c$,
\item $\m\nvDash\b$ or $\m\vDash\c$ for $\a=\b\to\c$,
\item $\m\vDash\b$ iff $\m\vDash\c$ for $\a=\b\leftrightarrow\c$,
\item $\m\vDash\b$ and $\m\vDash\c$ and $\langle\b,\c\rangle\in\r$ for $\a=\b\wi\c$,
\item ($\m\nvDash\b$ or $\m\vDash\c$) and $\langle\b,\c\rangle\in\r$ for $\a=\b\wto\c$.
\end{itemize}
Let $\m\in\M$, $\Sigma\subseteq\for$. We will write $\m\vDash\Sigma$ to indicate that for any $\varphi\in\Sigma$: $\m\vDash\varphi$. Let $\mathsf{K}\subseteq\M$. $\mathsf{K}\vDash\varphi$ means that for any $\m\in\K$ we have $\m\vDash\varphi$. As it is usually done, we will also understand $\vDash$ as a consequence relation. We say $\Sigma\vDash\varphi$ iff for any $\m\in\M$ we have $\m\vDash\Sigma$ implies $\m\vDash\varphi$.
We say that $\r\vDash\a$ iff for any $v$: $\model\vDash\a$. Similarly $\r\vDash\Sigma$ iff for any $\sigma\in\Sigma$: $\r\vDash\sigma$. Let $\R\subseteq\mathcal{P}(\for^2)$ be a set of relations. We say that $\R\vDash\a$ iff for any $\r\in\R$ we have $\r\vDash\a$. Let $\Sigma\subseteq\for$, $\a\in\for$ and $\R$ be some set of relations. We say that $\Sigma\vDash_\R\a$ iff for any $\r\in\R$, any $\m=\model$: $\m\vDash\Sigma$ implies $\m\vDash\a$. One can see an analogy between Epstein relations and the so called frames in modal logic. Note an important difference though: unlike frames, Epstein relations do not necessarily have theories that are closed under substitutions. To give an example, let $\r=\{\langle p,p\rangle\}$. It is easy to see that $\r\vDash p\wto p$ but $\r\nvDash q\wto q$. The difference between the theory of a single model and a single relation theory is that the first one is complete in the sense that for any $\a\in\for$: $\m\vDash\a$ or $\m\vDash\neg\a$, while it is not the case for the latter. Any propositional letter gives an immediate example: both $\r\nvDash p$ and $\r\nvDash\neg p$, for any $\r$.

In \cite{jarkacz} the consequence relation $\vDash$ has been called $\rF$. The authors constructed an adequate tableau deductive system for this consequence relation. Here, we will focus on axiomatic system. Let $\cpl$ be the set of all classical tautologies. Define $\rF$ to be the least set closed under uniform substitution and modus ponens (MP) $\{p, p\to q\}/q$ containing $\cpl$ and two additional axioms: 
\begin{enumerate}
    \item $(p\wto q)\to (p\to q)$
    \item $(p\wi q)\leftrightarrow(p\wto q)\wedge(p\wedge q)$
\end{enumerate}
Let $\Lambda\subseteq\for$. By $\rF\Lambda$ we will denote least sets of formulas that contain $\rF$ and $\Lambda$ which are closed under uniform substitution and MP. Later on, we will refer to such sets as logics. When $\Lambda=\{\varphi\}$ we will write $\rF\varphi$ to simplify notation. 

Analogously to the semantic consequence relation, we define the syntactic notion of derivability. Let $\Sigma\cup\{\varphi\}\subseteq\for$. We say that $\varphi$ is derivable (provable) from $\Sigma$ (shortly $\Sigma\vdash\varphi$) iff there is a standard Hilbert proof of $\varphi$ from $\Sigma$, i.e. finite sequence of formulas whose last element is $\varphi$ and all elements are either members of $\Sigma\cup\rF$, or are obtained from previous ones by means of MP. It is easy to see that for any $\Lambda\subseteq\for$ we have $\varphi\in\rF\Lambda$ iff $\{\sigma(\psi):\psi\in\Lambda,\sigma\in\subst\}\vdash\varphi$.

It is an exercise to check that $\vdash=\vDash$. $\subseteq$ is straightforward. For $\supseteq$ let $\Sigma^+$ be a maximally consistent set\footnote{Textbook Lindenbaum lemma can be used to show that each consistent set can be extended to a maximally consistent one.}. Let $\m=\model$ be such that for any $\psi\in\Phi$: $v(\psi)=1$ iff $\psi\in\Sigma^+$; and $\langle\varphi,\chi\rangle\in\r$ iff $\varphi\wto\chi\in\Sigma^+$. Inductive proof on the complexity of formulas shows that $\m\vDash\Sigma^+$. Base case and formulas built with boolean connectives are obvious. Let $\varphi\wto\chi\in\Sigma^+$. Then, by MP and axiom 1), also $\varphi\to\chi\in\Sigma^+$. Hence $\m\vDash\varphi\to\chi$ and $\langle\varphi,\chi\rangle\in\r$ which gives us $\m\vDash\varphi\wto\chi$. Now assume that $\varphi\wi\chi\in\Sigma^+$. Hence $\varphi\wedge\chi\in\Sigma^+$ and $\varphi\wto\chi\in\Sigma^+$ which means $\m\vDash\varphi\wedge\chi$ and $\langle\varphi,\chi\rangle\in\r$, so $\m\vDash\varphi\wto\chi$.
\section{$\SS$--set basic definitions}
In this section, we recall the notion of an $\SS$--set introduced in \cite{ja} and prove that each Epstein model has uncountably many equivalent Epstein models. 

Given a model $\m=\langle v,\r\rangle$, we define the theory of $\m$ to be the set: $\tthm=\{\a\in\for:\m\vDash\a\}$. Now let us define the relation of a theoretical equivalence.
\begin{definition}[Theoretical equivalence]\label{df model equivalence}
Given the models $\m=\langle v,\r\rangle$, $\n=\langle v',\r'\rangle$ we say that $\m$ and $\n$ are theoretically equivalent, symbolically: $\m\approx\n$ iff $\tthm=\th(\n)$. Obviously $\approx$ is an equivalence relation, hence for arbitrary model by $\eqm$ we shall denote the equivalence class: $\{\n\in\M:\m\approx\n\}$.
\end{definition}

In \cite{ja}, we have already introduced the $\SS$-set construction which enables us to generate the whole equivalence class of a given model with respect to relation $\approx$. Let us remind the construction. Let $\m=\model$. The Omega set of $\m$ is $\om=\{\langle\a,\b\rangle:\m\nvDash\a\to\b\}$. Having specified the Omega set of $\m$ we define the $\SS$-set of $\m$ to be $\sm=\{\nodel:v'=v\,\,and\,\,\r\setminus\om\subseteq\r'\subseteq\r\cup\om\}$. Let us recall that the following holds: 
\begin{lem}[Krawczyk, 2021]\label{f: S fact}
For any $\m=\model$ and $\n=\nodel$, we have: $\n\in\sm$ iff $\n\approx\m$. Hence $\sm=\eqm$. 
\end{lem}
Now, we shall prove that for any model $\m=\model$, its $\SS$-set $\sm$ is always uncountable:
\begin{thm}\label{t: cardinality of an S set}
For any $\m=\model$ the cardinality of $\sm$ is $2^{\aleph_0}$.
\end{thm}
\begin{proof}
First note that for any $\m\in\M$, $\sm$ is equinumerous with omega's powerset: $\mathcal{P}(\om)$. The mapping $X\mapsto \langle v,(\r\setminus\om)\cup X \rangle$ is a bijection from $\mathcal{P}(\om)$ onto $\sm$. It is easy to observe that $\om$ is countably infinite. Define $\neg^0\top:=\top$ and $\neg^{n+1}\top:=\neg\neg^n\top$. For any natural number $n$, $\m\nvDash\neg^{2n}\top\to\bot$, so $\om$ is infinite. Our language is countable hence $\om$ is of cardinality $\aleph_0$ and thus the cardinality of $\mathcal{P}(\om)$ is $2^{\aleph_0}$. 
\end{proof}

Theorem \ref{t: cardinality of an S set} yields an immediate corollary: that each model has uncountably many theoretically equivalent models.
\section{The van Benthem characterization analogon}

In the context of expressive power over Kripke frames, the van Benthem characterization theorem \cite{benthemphd, blackburn} -- put in a popular manner -- states that  modal languages are the bisimulation invariant fragment of first-order logic. The notion of a standard translation plays a crucial role in proving this theorem. Standard translation maps modal formulas into first order formulas of the corresponding first order language in a natural way: it reflects the truth conditions of a given modal formula. Modal operators naturally correspond to quantifiers, so first order logic is the obvious target of the standard translation mapping. Note though, that in the case of generalised Epstein logic, such a mapping does not seem that obvious. There are no connectives which necessarily impose using quantifiers to express their truth conditions. Nonetheless, the use of atomic first order sentences comes quite naturally i.e. $St(\a\wto\b)=St(\a\to\b)\wedge\mathcal{R}(\a,\b)$, where $\mathcal{R}$ is the binary relational symbol corresponding to $\r$. However, an atomic first order sentence $\mathcal{R}(\a,\b)$ can be treated as an atomic expression of classical propositional logic. If we were to decode such a sentence by indexing it with a pair of related elements, it would have the following form: $p_{\langle\a,\b\rangle}$. Now we can define the corresponding classical propositional language. Let $\at$ be the set $\Phi\cup\Phi^\mathcal{R}$, where $\Phi^\mathcal{R}=\{p_{\langle\a,\b\rangle}:\a,\b\in\for\}$. One can easily see that $\Phi^\mathcal{R}$ is countable, so our extended language remains countable. The set of $\cpl$ formulas $\for^\ast$ is the set generated from $\at$ and connectives: $\neg,\wedge,\lor,\to,\leftrightarrow$. 
Each Epstein model $\m$ can be seen as a model of $\cpl$ defined on the set of formulas $\for^\ast$. We say that $\m\models_\cpl p_n$ iff $v(p_n)=1$ and $\m\models_\cpl p_{\langle\a,\b\rangle}$ iff $\langle\a,\b\rangle\in\r$. Definition can be extended to arbitrary formulas in a standard way. For $A\in\for^\ast$, we will write $\models_\cpl A$ to indicate the fact that for any $\m=\model$ we have $\m\models_\cpl A$. Now we are ready to define the standard translation for the language of Epstein logic. $St:\for\longrightarrow\for^\ast$. 
\begin{itemize}
\item $St(p_n)=p_n$, 
\item $St(\neg\a)=\neg St(\a)$, 
\item $St(\a\ast\b)=St(\a)\ast St(\b)$, for $\ast\in\{\lor,\wedge,\to,\leftrightarrow\}$. 
\item $St(\a\wto\b)=St(\a\to\b)\wedge p_{\langle\a,\b\rangle}$ 
\item $St(\a\wi\b)=St(\a\wedge\b)\wedge p_{\langle\a,\b\rangle}$.
\end{itemize}

Obviously, the following holds for any $\m=\model$: $\m\vDash\a$ iff $\m\models_\cpl St(\a)$. Now we can ask the following question: when is a formula of $\cpl$ equivalent to a standard translation of some Epstein formula? The answer is given by the characterization theorem. Let us first define the $\SS$-set invariance for $\cpl$ formulas.
\begin{definition}
We say that a formula $A\in\for^\ast$ is $\SS$-set invariant iff for any $\m=\model$, $\n=\nodel$, $\m\models_\cpl A$ and $\n\in\sm$ implies $\n\models_\cpl A$. 
\end{definition}
Now we can move on to the characterization theorem.
\begin{thm}
Let $A\in\for^\ast$. There is $\a\in\for$ such that $\models_\cpl A\leftrightarrow St(\a)$ iff $A$ is $\SS$-set invariant.
\end{thm}

\begin{proof}
The left to right direction follows from the $\SS$-set Lemma \ref{lm: S lemma} and from the previously mentioned fact that for any $\m=\model$: $\m\vDash\a$ iff $\m\models_\cpl St(\a)$. For the right to left, assume that $A\in\for^\ast$ is $\SS$-set invariant. Let $RC(A)=\{B:B=St(\b)\,\,for\,\,some\,\,\b\in\for\,\,and\,\,A\models_\cpl B\}$. If $A$ is inconsistent, then $A$ is equivalent to $St(p_0\wedge\neg p_0)$. Assume then, that $A$ is consistent. This means that $RC(A)$ is consistent. Let $\m=\model$ be such that $\m\models_\cpl RC(A)$. Let $T(\m)=\{B:B=St(\b)\,\,for\,\,some\,\,\b\in\for\,\,and\,\,\m\models_\cpl B\}$. $T(\m)\cup\{A\}$ is consistent. If it was not, then $A\not\models_\cpl X$ for some $X=\{C_1,...,C_n\}\subseteq T(\m)$. This would mean that $\neg C_1\lor...\lor \neg C_n\in RC(A)$ and also $\neg C_1\lor...\lor \neg C_n\in T(\m)$ which leads to a contradiction. Let $\n=\nodel$ be such that $\n\models_\cpl T(\m)\cup\{A\}$. Since $\n\models_\cpl T(\m)$, we  know that $\n\vDash \th(\m)$, so $\th(\m)\subseteq \th(\n)$. Each theory of a model is maximally consistent, so we immediately can infer the opposite inclusion which gives $\n\approx\m$. This means that $\m\in \SS^\n$. We assumed that $A$ is $\SS$-set invariant so $\m\models_\cpl A$. This means that $RC(A)\models_\cpl A$. By compactness of $\cpl$, we obtain $\models_\cpl B_1\wedge...\wedge B_n\leftrightarrow A$, where $\{B_1,...,B_n\}\subseteq RC(A)$. Each $B_i=St(\b_i)$, so $\models_\cpl A\leftrightarrow St(\b_1\wedge...\wedge\b_n)$.
\end{proof}

\section{Epstein incompleteness}
It was a significant breakthrough in the theory of modal logic when examples of Kripke-incomplete logics (the ones with no adequate class of Kripke frames) were found \cite{Thomason, Fine, benthem}. This discoveries culminated in a shocking Blok's result \cite{blok} where uncountably many normal modal logics with no adequate classes of Kripke frames were proven to exist. We will prove an analogous theorem for Epstein relations. However, to make the reasoning clear, we shall proceed step by step. Since the phenomenon of Epstein-incompletess is terra incognita, we shall start with an example of continuum many logics which are Epstein complete. Then, we will proceed to give a single example of Epstein-incomplete logic (logic with no adequate set of Epstein relations). 
Then, we will combine the techniques used in previous two results to prove the main theorem of this section sating that there are uncountably many Epstein incomplete axiomatic extensions of $\mathcal{F}$.

Following the preliminary section, by a logic we understand a set of formulas containing $\mathcal{F}$, closed under uniform substitution and Modus Ponens. A logic $\lambda$ is Epstein complete iff there is a set $X\subseteq\mathcal{P}(\for^2)$ of Epstein relations such that $\lambda=\{\varphi: X\vDash\varphi\}$. Otherwise, a logic is said to be Epstein-incomplete.

Let us start with an Epstein-completeness result.
\begin{thm}\label{t: uncountably many complete extensions}
    There are $2^{\aleph_0}$ Epstein complete extensions of $\mathcal{F}$. 
\end{thm}
\begin{proof}
    Let $p^0:=p$. Define $p^{n+1}:=p\to p^n$. Now let $K_\omega=\{p\wto p^n:n\in\omega\}$. Hence elements of $K_\omega$ are of the form $p\wto (p\to p)(=p\wto p^1)$, $p\wto (p\to(p\to p))(=p\wto p^2)$, and so on. Let $T\subseteq \omega$. Let $K_T=\{p\wto p^k:k\in T\}$. Hence $K_T\subseteq K_\omega$. To prove the theorem, we will show that for each two non-empty subsets of positive natural numbers $T,V\subseteq\omega\setminus\{0\}$ such that $T\neq V$ we have $\mathcal{F}K_T\neq\mathcal{F}K_V$. Let $T,V$ be such subsets. Without the loss of generality assume that $n\in T\setminus V$. Let $\m=\model$ be such that $\r=(\for\times\for)\setminus\{\langle p, p^n \rangle\}$. Let $0<k\neq n$. It is obvious that $\m\vDash \sigma(p^k)$ for any $\sigma\in\subst$, since $p^k$ is a tautology. Also $\langle \sigma(p),\sigma(p^k)\rangle\in \r$ since otherwise we would get $n=k$, contradiction. Thus $\m\vDash\sigma(p\wto p^k)$. Thus we have shown that $\m\vDash\{\sigma(p\wto p^k):k\neq n, \sigma\in\subst\}$ and thus obviously $\m\vDash \{\sigma(\psi):\psi\in K_V\}$, since $\{\sigma(\psi):\psi\in K_V\}\subseteq \{\sigma(p\wto p^k):k\neq n, \sigma\in\subst\}$. But also $\m\nvDash p\wto p^n$, so $\{\sigma(\psi):\psi\in K_V\}\nvDash p\wto p^n$ which means $\mathcal{F}K_V\neq\mathcal{F}K_T$.
    
    Now we need to show that each $\mathcal{F}K_T$ is Epstein-complete. Let $\r_0^T=\{\langle \sigma(p),\sigma(p^k)\rangle:\sigma\in\subst ,k\in T\}$. Define $\R^T:=\{\r \subseteq\for^2: \r_0^T\subseteq\r\}$. We claim that $\mathcal{F}K_T=\{\varphi\in\for:\R^T\vDash\varphi\}$. The ($\subseteq$) inclusion is a matter of a routine check. For the opposite inclusion let $\Sigma^+$ be maximally consistent set which includes $\mathcal{F}K_T$ as a subset. It is easy to observe that a model $\m=\model$ of $\Sigma^+$ defined in the same way as it is done at the end of section 2, is such that $\r_0^T\subseteq\r$, i.e. $\r\in\R^T$. This proves the opposite inclusion. 
\end{proof}
Thus, we already know that there is an abundance of logics which have adequate sets of Epstein relations.
Now we begin our exploration of incompleteness phenomenon. Let us start with a single example of Epstein-incomplete logic. Let $\alpha:= p\to(q\wto p)$.
\begin{lem}\label{l: single icomplete logic formula}
    For any Epstein relation $\r$ such that $\r\vDash\{\sigma(\alpha):\sigma\in\subst\}$ we also have $\r\vDash p \wto p$.
\end{lem}
\begin{proof}
    Assume that $\r\nvDash p\wto p$. Hence it must be the case that $\langle p,p\rangle\notin\r$. Take substitution $\sigma$ s.t. $\sigma(q)=p$ and for the remaining propositional letters $\varphi\neq q$, $\sigma(\varphi)=\varphi$. Then $\sigma(\alpha)=p\to(p\wto p)$. Take valuation $v$ s.t. for any propositional letter $\varphi$, $v(\varphi)=1$. Then $\langle v, \r\rangle\nvDash p\to(p\wto p)$, so $\r\nvDash p\to(p\wto p)$. Hence $\r\nvDash\{\sigma(\alpha):\sigma\in\subst\}$.
    \end{proof}
\begin{lem}\label{l: single incomplete logic does not contain}
    $p\wto p\notin\mathcal{F}\alpha$.
\end{lem}
\begin{proof}
It is enough to show that $\{\sigma(\alpha):\sigma\in\subst\}\nvDash p\wto p$. Let $\m=\model$ where $\r=\for^2\setminus\{\langle p,p\rangle\}$ and $v(\varphi)=0$ for each propositional letter $\varphi$. Let $\psi\in\for$ be any formula s.t. $\psi\neq p$. Assume $\m\vDash\psi$. Hence also $\chi\to\psi$ for any $\chi\in\for$. Also $\langle \chi,\psi\rangle\in\r$, so $\m\vDash\psi\to(\chi\wto\psi)$. If, on the other hand, $\psi= p$, then $\m\nvDash p$, so also $\m\vDash p\to (p\wto p)$. This gives us $\m\vDash\{\sigma(\alpha):\sigma\in\subst\}$. But obviously $\m\nvDash p\wto p$, so $\{\sigma(\alpha):\sigma\in\subst\}\nvDash p\wto p$ which gives the result.
\end{proof}
\begin{thm}
    $\mathcal{F}\alpha$ is Epstein incomplete.
\end{thm}\label{t: incomplete logic example}
\begin{proof}
    Assume that there is a set of relations $\R$ such that $\mathcal{F}\alpha=\{\varphi: \forall \r\in\R\,\,\, \r\vDash\varphi\}$. Then by Lemma \ref{l: single icomplete logic formula} $p\wto p\in \mathcal{F}\alpha$. But by Lemma \ref{l: single incomplete logic does not contain} $p\wto p\notin \mathcal{F}\alpha$. Contradiction, so $\mathcal{F}\alpha$ is Epstein incomplete.
\end{proof}

Now we are ready to extend our previous techniques to show that there are exactly uncountably many extensions of $\mathcal{F}$ which are Epstein incomplete.
First we construct a sequence of formulas. 
\begin{align*}
    p^0 &:= q\wto p\\
    p^{n+1} &:= p\to p^n
\end{align*}
We have $p^1=p\to(q\wto p)$ and $p_2$ is $p\to(p\to(q\wto p))$ etc. Now we define a set of formulas $\Lambda_\omega:=\{q\wto p^n:  n\in\omega\}$.
\begin{lem}\label{l: continuum incomplete R satisifes}
For any $n\in\omega\setminus\{0\}$, any Epstein relation $\r$, if $\r\vDash\{\sigma(q\wto p^n):\sigma\in\subst\}$, then $\r\vDash p\wto p$.
\end{lem}
\begin{proof}
    Let $\sigma$ be a substitution such that $\sigma(q)=p$ and for the rest of propositional letters $\sigma(\varphi)=\varphi$. Hence $\sigma(q\wto p^n)$ is of the form $p\wto \underbrace{(p\to(p\to...}_{\text{n times}}(p\wto p)...)$. Let $\r\vDash\sigma(q\wto p^n)$.  Obviously, it must be the case that $\langle p, \sigma(p^n)\rangle\in \r$. Assume that $\langle p,p\rangle\notin \r$. Let $v(p)=1$. Obviously $\model\nvDash p\wto p$, hence $\model\nvDash p\to (p\wto p)$, and $\model\nvDash p\to (p\to (p\wto p))$, and so on. Thus $\model\nvDash \sigma(p^n)$ for any $n\in\omega$. Hence also $\model\nvDash p\wto \sigma(p^n)$ for any $n\in\omega$ which gives us contradiction. Hence $\langle p,p\rangle\in\r$ and for this reason $\r\vDash p\wto p$.
\end{proof}
\begin{lem}\label{l: continuum incomplete formula notin}
Let $S$ be a non-empty subset of $\omega$. Then $\{\sigma(q\wto p^n):n\in S, \sigma\in\subst\}\nvDash p\wto p$. 
\end{lem}
\begin{proof}
    Let $\m=\model$ be such that $\r=\for\times\for\setminus\{\langle p,p\rangle\}$ and $v(p)=0$. Let $\psi\in\for$ and $\psi\neq p$. For any $\varphi\in\for$, we have $\model\vDash \varphi\to (\psi\wto\varphi)$. To see that assume $\model\vDash\varphi$. Hence also $\model\vDash\psi\to\varphi$ and since $\langle\psi,\varphi\rangle\in\r$, we get $\model\vDash\psi\wto\varphi$. Straightforward inductive step shows that we get for any $n\in\omega$, $\model\vDash\varphi\to\varphi^n$ and since $\langle \psi,\varphi^n\rangle\in\r$ it is also the case that $\model\vDash\psi\wto\varphi^n$. Now let $\psi = p$. If $\varphi\neq p$, the reasoning from the previous case applies since then $\langle p, \varphi\rangle\in \r$. Now let $\varphi=p$. Then for each formula of the form $p\to(p\to...(p\wto p)...)$ we have $\m\vDash p\to(p\to...(p\wto p)...)$ since $\model\nvDash p$. Observe that $p$ is not a formula of the above form -- see definition of $\varphi^n$. Thus we also have $\langle p, p\to(p\to...(p\wto p)...)\rangle\in\r$ which means $\m\vDash p\wto (p\to(p\to...(p\wto p)...)$. Thus we get $\m\vDash\{\sigma(p\wto p^n):n\in \omega, \sigma\in\subst\}$ and since $S\subseteq\omega$ also $\m\vDash\{\sigma(p\wto p^n):n\in S, \sigma\in\subst\}$. But also $\m\nvDash p\wto p$ since $\langle p,p\rangle\notin\r$ which proves the lemma.
\end{proof}
For our final lemma, define the set $\Lambda_T:=\{q\wto p^m:m\in T\}$, for each $T\subseteq\omega\setminus\{0\}$. Obviously, we have $\Lambda_T\subseteq \Lambda_\omega$.
\begin{lem}\label{l:continuum incomplete each different}
    Let $P$ and $T$ be a non-empty subsets of $\omega\setminus\{0\}$ such that $P\neq T$. Then $\mathcal{F}\Lambda_P\neq\mathcal{F}\Lambda_T$. 
\end{lem}
\begin{proof}
    Without the loss of generality, assume that there is $k\in P\setminus T$. Let $\m=\model$ be such that $\r=\for\times\for\setminus\{\langle q, p^k\rangle\}$ and for any propositional letter $\varphi$, $v(\varphi)=0$. Observe that for any $n\neq k$, any $\sigma\in\subst$, $\m\vDash \sigma(q\wto p^n)$. To see that, let $0\neq n\neq k$ and $\sigma\in\subst$. We will consider two cases. 

    If $\sigma(q)=q$, then $\m\nvDash \sigma(q)$ and so $\m\vDash \sigma(q\to p^n)$. It is obvious that $\sigma(p^n)\neq(p^k)$ so $\langle \sigma(q),\sigma(p^n)\rangle\in\r$ and thus $\m\vDash \sigma(q\wto p^n)$.

    Let $\sigma(q)\neq q$. Then $\langle\sigma(q),\sigma(p)\rangle\in\r$ and by a straightforward induction we get $\m\vDash\sigma(p^n)$. Also $\m\vDash\sigma(q)\to\sigma(p^n)$ and $\langle \sigma(q),\sigma(p^n)\rangle\in\r$, so $\m\vDash\sigma(q)\wto\sigma(p^n)$.

    We have shown that $\m\vDash\sigma(q\wto p^n)$ for each $1\leq n\neq k$ and thus also $\m\vDash\{\sigma(\varphi):\varphi \in \Lambda_T, \sigma\in\subst\}$. Obviously $\m\nvDash q\wto p^k$, which means that $q\wto p^k\notin \mathcal{F}\Lambda_T$. We conclude that $\mathcal{F}\Lambda_S\neq\mathcal{F}\Lambda_T$.
\end{proof}

Now we can prove the theorem.
\begin{thm}
    There are $2^{\aleph_0}$ Epstein incomplete extensions of $\mathcal{F}$.
\end{thm}
\begin{proof}
    For any non-empty $S\subseteq\omega\setminus\{0\}$, we have $p\wto p\notin \mathcal{F}\Lambda_S$ by Lemma \ref{l: continuum incomplete formula notin}. Also for each Epstein relation $\r$ s.t. $\r\vDash\mathcal{F}\Lambda_S$ we have $\r\vDash p\wto p$ by Lemma \ref{l: continuum incomplete R satisifes}, so $\mathcal{F}\Lambda_S$ is Epstein incomplete for each $\emptyset\neq S\subseteq\omega\setminus\{0\}$. By Lemma \ref{l:continuum incomplete each different} we have uncountably many such logics.
\end{proof}

\section{The two conditions: undefinability and strong completeness}
In order to show the power of our methods, we choose to analyze two exemplary conditions: symmetry and the one that may be called predecessor -- negation elimination. Both of them, among some others, are imposed by Epstein on the models of his relatedness logics \cite{epstein90}. 
\begin{itemize}
\item[(s)]: Symmetry: for any $\a,\b$: if $\a\r\b$, then $\b\r\a$,
\item[(n)]: for any $\a,\b$: if $\neg\a\r\b$ then $\a\r\b$.
\end{itemize}

Observe that both of these conditions are undefinable. We will extend the technique already signalised in \cite{ja} to prove their undefinability. In each case we will show that neither the set of Epstein relations fulfilling the respective condition is definable nor is the set of Epstein models meeting the condition.

\begin{definition}
Let $\K$ be some set of Epstein relations, $\Gamma\subseteq\for$. We say that $\Gamma$ defines $\K$ when the following holds: for any Epstein relation $\r$, $\r\vDash\Gamma$ iff $\r\in\K$. We say that the set of relations $\K$ is definable iff there is $\Gamma\subseteq\for$ such that $\Gamma$ defines $\K$. Analogously, we define the definability of a given set of Epstein models.
\end{definition}




\begin{prop}\label{f: relation undefinability}
Let $\R^s$ be the set of all symmetric relations, $\R^n$ the set of all relations fulfilling $n$ condition. Both $\R^s$ and $\R^n$ are undefinable.
\end{prop}
\begin{proof}For symmetry assume that $\R^s$ is definable. Let $\Gamma$ define $\R^s$. Let $\r=\lbrace\langle p,q\rangle, \langle q,p\rangle\rbrace$. $\r$ is symmetric, so $\r\vDash\Gamma$. This means that for any $v$, any $\m=\model$ we have $\m\vDash\Gamma$. Let $v'$ be arbitrary and  $\m'=\langle v',\r\rangle$. Obviously $\m'\nvDash\top\to\bot$. Let $\r'=\r\cup\{\langle\top,\bot\rangle\}$. It is easy to see that $\langle \top,\bot\rangle\in\Omega^{\m'}$ (actually, $\langle \top,\bot\rangle$ is in $\Omega^{\m}$ for any $\m$), so $\r'\subseteq\r\cup\Omega^{\m'}$. Also $\r\setminus\Omega^{\m'}\subseteq\r'$. Let $\m''=\langle v',\r'\rangle$. Valuation is the same, so $\m''\in\SS^{\m'}$, hence $\m''\vDash\Gamma$. Note that $v'$ was arbitrary hence $\r'\vDash\Gamma$. But $\langle\bot,\top\rangle\notin\r'$, so $\r'$ is not symmetric! Contradiction.

Assume that $\R^n$ is definable. Let $\Delta$ define $\R^n$. Let $\r=\{\langle\neg p, q\rangle, \langle p,q\rangle\}$. Obviously $\r\in\R^n$. This means that $\r\vDash\Delta$. Let $v$ be arbitrary valuation. We know that $\m=\model\vDash\Delta$ and also that $\m\nvDash\neg\bot\to\bot$. This means that $\langle\neg\bot,\bot\rangle\in\om$. Let $\n=\nodel$ where $v'=v$ and $\r'=\r\cup\{\langle\neg\bot,\bot\rangle\}$. It is easy to see that $\r\setminus\om\subseteq\r'\subseteq\r\cup\om$, so $\n\vDash\Delta$. But $v=v'$ was arbitrary hence $\r'\vDash\Delta$. But $\langle\bot,\bot\rangle\notin\r'$ so $\r'\notin\r^n$. Contradiction. \end{proof}

Observe that we can use similar argument to prove:
\begin{cor}\label{f: model undefinability}
Let $\M^s$ be the set of symmetric models, $\M^n$ be the set of models with (n)-relations  i.e. $\M^s=\{\langle v,\r\rangle:\r\,\,\text{is symmetric}\}$, $\M^n=\{\langle v,\r\rangle:\neg\a\r\b\Longrightarrow\a\r\b\}$.  Both $\M^s$ and $\M^n$ are not definable.
\end{cor}
\begin{proof}
Adding valuations will not affect the argument in any way, since the reasoning is based on $\top$ and $\bot$.
\end{proof}  
Further we have:
\begin{prop}
$\M^s\cap\M^n$ and $\R^s\cap\R^n$ are undefinable.
\end{prop}
\begin{proof}
Assume that $\M^s\cap\M^n$ is definable. Let $\Gamma$ define $\M^s\cap\M^n$. Let $\m=\model$, where $\r=\emptyset$. Empty relation is both symmetric and meets the (n) condition, so $\m\in\M^s\cap\M^n$. This means $\m\vDash\Gamma$. Let $\n=\langle v, \r'\rangle$ where $\r'=\langle\top,\bot\rangle$. Obviously $\r'\subseteq\r\cup\om$, so $\n\vDash\Gamma$. But $\n\notin\M^s\cap\M^n$, since $\n$ is not symmetric.

To show that $\R^s\cap\R^n$ is undefinable we reason in the same manner, assuming that $\R^s\cap\R^n$ is definable by some set of formulas $\Sigma$. We notice that $\r=\emptyset\in\R^s\cap\R^n$. We use \eqref{f: S fact} to show that $\r'=\langle\top,\bot\rangle$ satisfies the defining set of formulas $\Sigma$ which immediately leads us to contradiction.
\end{proof}

\subsection{The $\SS$ set -- syntactical version}

In order to obtain $\sm$ we have to know the structure of a model $\m$. It can be the case, though, that we do not know the structure of a given model $\m$, but we want to construct the set $\eqm$ knowing only its theory. This is precisely the case in standard completeness proofs, where a model is constructed from maximally consistent set of formulas. Now we shall introduce the way of obtaining $\eqm$ from the theory of $\m$, which will prepare us for the later completeness proofs. Assume that we know the theory $\tthm$ of some model $\m$. 
\begin{definition}\label{df S set}We define two relations in the following way:	
	\begin{align*}
		\rthmin &=\{\langle\a,\b\rangle: \a\wto\b\in\tthm\},\\
		\rthmax &=\{\langle\a,\b\rangle: \a\wto\b\in\tthm\,\,or\,\,\a\to\b\notin\tthm\}.
	\end{align*}
We also define the valuation $\vthm=\Phi\longrightarrow\{0,1\}$ win an obvious manner: $\vthm(\a)=1$ iff $\a\in\tthm$, for any $\a\in\Phi$. Here there is no surprise -- valuations are uniquely determined by the theories of respective models (unlike relations).
Now we define the $\Sthm$ set, the set of models similar to $\m$:
\[\Sthm=\{\n=\nodel:v'=\vthm\,\,\&\,\,\rthmin\subseteq\r'\subseteq\rthmax\}\]
\end{definition}

\begin{lem}\label{lm: S lemma}
For any $\m=\model$ the following equality holds: 
\[ \Sthm=\{\n=\nodel:\tthm=\thn\}\]

\end{lem}
\begin{proof}
Let $\m=\model$ be arbitrary. To prove the left to right inclusion, let $\n=\nodel\in\Sthm$. First, we will show that for any $\a\in\tthm$ we have $\n\vDash\a$. Assume that $\a\in\tthm$. Let $\a\in\Phi$. By \eqref{df S set} we know that $\vthm=v'$, hence $\n\vDash\a$. Let $\a=\neg\b$. $\neg\b\in\tthm$ iff $\b\notin\tthm$. By inductive hypothesis $\n\nvDash\b$,  which means $\n\vDash\neg\b$. Let $\a=\b\lor\c$. $\b\lor\c\in\tthm$ iff $\b\in\tthm$ or $\c\in\tthm$. By inductive hypothesis $\n\vDash\b$ or $\n\vDash\c$, so $\n\vDash\b\lor\c$. The proof for the rest of boolean cases goes in a similar manner. Assume that $\a=\b\wto\c$. $\b\wto\c\in\tthm$, so $\b\notin\tthm$ or $\c\in\tthm$, which by inductive hypothesis gives us  $\n\nvDash\b$ or $\n\vDash\c$. By \eqref{df S set} we know that $\langle\b,\c\rangle\in\rthmin$ and $\rthmin\subseteq\r'$, which means that $\langle\b,\c\rangle\in\r'$. So $\n\vDash\b\wto\c$. Let $\a=\b\wi\c$. $\b\wi\c\in\tthm$ means $\b\in\tthm$ and $\c\in\tthm$. By inductive hypothesis we obtain $\n\vDash\b$ and $\n\vDash\c$. We also know that $\b\wto\c\in\tthm$, since $\m\vDash\b\wi\c$ implies $\m\vDash\b\wto\c$. This means that $\langle\b,\c\rangle\in\rthmin$. We know that $\rthmin\subseteq\r'$, so $\n\vDash\b\wi\c$.  This way, we have shown that $\n\vDash\tthm$.

To show the opposite inclusion, assume that $\n=\nodel\notin\Sthm$. Hence at least one of the following must hold 1) $v'\neq\vthm$ or 2) $\rthmin\nsubseteq\r'$ or 3) $\r'\nsubseteq\rthmax$. Let 1) hold. This means that for some $\a\in\Phi$ we have $\n\vDash\a$ and $\a\notin\tthm$ or  $\n\nvDash\a$ and $\a\in\tthm$. In both cases we get $\thn\neq\tthm$. Suppose 2) is the case. This means that there is $\langle\a,\b\rangle\in\for\times\for$ such that $\a\wto\b\in\tthm$ and $\langle\a,\b\rangle\notin\r'$. This means that $\n\nvDash\a\wto\b$, so $\thn\neq\tthm$. Finally assume that 3) holds. This means that there is $\langle\a,\b\rangle\in\r'$ such that $\a\wto\b\notin\tthm$ and $\a\to\b\in\tthm$. Now we have two possibilities: either $\n\nvDash\a\to\b$ or $\n\vDash\a\to\b$. If the first disjunct is true, then $\tthm\neq\thn$ follows immediately. Thus, assume the second disjunct holds. This way, we obtain $\n\vDash\a\wto\b$. But $\a\wto\b\notin\tthm$, so again $\tthm\neq\thn$.
\end{proof}



\subsection{Axiomatic systems for undefinable sets of Epstein structures}
Despite the undefinability results stated in facts \eqref{f: model undefinability} and \eqref{f: relation undefinability}, we claim that it is possible to give an  adequate axiomatic systems which generate the consequence relations of $\R^n$, $\R^s$ and $\R^s\cap\R^n$. 
To make the notions already introduced in the preliminary section a little more precise, we shall start from recalling some textbook definitions of Hilbert-style proof, (maximal) consistency and so on. By $\cpl$ we mean all classical tautologies in the language of Epstein logic.
\begin{enumerate}
\item $\cpl$
\item $(p\wto q)\to(p\to q)$
\item $(p\wto q)\wedge(p\wedge q)\leftrightarrow(p\wi q)$
\end{enumerate}
again, the rules of the system are Modus Ponens and substitution. $\rF\subseteq\for$ is the least set of formulas closed under MP and substitution. Let $\Lambda\subseteq\for$. $\rF\Lambda$ is the least set of formulas containing $\rF\cup\Lambda$ which is closed under the two above-mentioned rules.
\begin{definition}[$\rF\Lambda$-proof]\label{df: proof} Let $\Sigma\subseteq\for$, $\a\in\for$. We will define the notion of a proof in a standard way for normal Epstein axiomatic systems.
The $\rF\Lambda$-proof of a formula $\varphi$ from the set of premises $\Sigma$ is a sequence $\psi_1,...,\psi_n$, where $\psi_n=\varphi$, and for each $i\in\lbrace 1,...,n\rbrace$, $\psi_i$ fulfils at least one of the following:
\begin{align}
\psi_i\in\rF\Lambda\cup\Sigma, or\\
\text{there are }j,k<i,\text{ such that }\psi_j=\psi_k\to\psi_i.
\end{align}
When there is an $\rF\Lambda$-proof of $\a$ from $\Sigma$ ($\a$ is provable from $\Sigma$), we will write $\Sigma\vdash_{\rF\Lambda}\a$. Observe that $\emptyset\vdash_{\rF\Lambda}\a$ means $\a\in\rF$. By theses of $\rF\Lambda$ we mean elements of $\rF$. We will shortly write $\vdash_{\rF\Lambda}\a$ instead $\emptyset\vdash_{\rF\Lambda}\a$.
\end{definition}

\begin{definition}[Maximal consistent set]  Let $\Gamma$ be a set of formulas. $\Gamma$ is $\rF\Lambda$-consistent iff $\Gamma\nvdash_{\rF\Lambda}\bot$; otherwise, it is inconsistent. $\Gamma$ is $\rF\Lambda$-maximally consistent set ($\rF\Lambda$-mcs) iff it is $\rF\Lambda$-consistent and all of its proper extensions are inconsistent.
\end{definition}
\noindent Using the well-known Lindenbaum construction, we can prove the following:
\begin{lem}{(Lindenbaum's Lemma)}\label{lind}
\\Any $\rF\Lambda$-consistent set of formulas $\Sigma$ can be extended to $\rF$-maximal consistent set ($\rF\Lambda$-mcs).
\end{lem}

Each maximally consistent set enjoys some standard properties:
\begin{prop}\label{f: mcsfact}
The following hold for any $\rF\Lambda-mcs$ $\Sigma^+$:
\begin{align*}
\text{if } \vdash_{\rF\Lambda}\a &\text{ then } \a\in\Sigma^+,\\
\neg\a\in\Sigma^+ &\text{ iff } \a\notin\Sigma^+,\\
\a\lor\b\in\Sigma^+ &\text{ iff }\a\in\Sigma^+ \text{ or }\b\in\Sigma^+,\\
\varphi , \psi\in\Sigma^+ &\text{ iff } \varphi \wedge \psi\in\Sigma^+,\\
\a\to\b\in\Sigma^+ &\text{ iff } \a\notin\Sigma^+ \text{ or }\b\in\Sigma^+,\\
\a\leftrightarrow\b &\text{ iff } \a\in\Sigma^+ \text{ if and only if }\b\in\Sigma^+,\\
\text{if }\varphi\in\Sigma^+ \text{ and }\vdash_{\rF\Lambda}\varphi\to\psi, &\text{ then }\psi\in\Sigma^+.
\end{align*}
\end{prop}

Each fact can be easily proven from the construction of $\rF\Lambda-mcs$ and the fact that $\rF\Lambda$ is closed under classical logic.

\begin{prop}[Set of all models for a given mcs]\label{f: canonical set}
Let $\rF\Lambda$ be Epstein logic and let $\splus$ be $\rF\Lambda-mcs$. The exact set of Epstein models which verify $\splus$ is $\Ssplus=\{\model:v=\vsplus,\,\,\rsplusmin\subseteq\r\subseteq\rsplusmax\}$, where $\vsplus$, $\rsplusmin$, $\rsplusmax$ are defined in an analogous way as it is done in \eqref{df S set}.
\end{prop}
Using the fact \eqref{f: mcsfact}, the proof can be given in a similar way as in case of lemma $\eqref{lm: S lemma}$.

Before we move on, we need to recall some results from \cite{klon20}. Klonowki's thesis concerns generalised Epstein logics with even more Epstein-style connectives than \cite{jarkacz}. He studies the language with basic boolean connectives $\neg,\lor,\wedge,\to,\leftrightarrow$ and Epstein style ones: $\lor^\mathcal{E},\wedge^\mathcal{E},\to^\mathcal{E},\leftrightarrow^\mathcal{E}$ interpreted in the standard Epstein way: $\model\vDash\varphi\ast^\mathcal{E}\psi$ iff $\model\vDash\varphi\ast\psi$ and $\langle\varphi,\psi\rangle\in\r$. As Klonowski observes, a language this rich is capable of expressing Epstein relation in the sense that: $\langle\varphi,\psi\rangle\in\r$ iff $\r\vDash (\varphi\to^\mathcal{E}\psi)\lor(\varphi\lor^\mathcal{E}\psi)$. Define $\varphi\mathcal{R}\psi:=(\varphi\to^\mathcal{E}\psi)\lor(\varphi\lor^\mathcal{E}\psi)$ \cite{klon20}. This way, axiomatization becomes trivial: basic system is $\cpl$ with axioms of the form $(p\ast^\mathcal{E} q)\leftrightarrow (p\ast q)\wedge (p\mathcal{R} q)$, for each $\ast\in\{\lor,\wedge,\to,\leftrightarrow\}$. Obviously Modus Ponens and substitution are added as only rules. The cases under consideration are equally trivial: for symmetric relations, we add $p\mathcal{R}q\to q\mathcal{R}p$. The reader sees immediately how it goes for the other cases -- it is a matter of rewriting the condition. Thus, we claim that weaker Epstein-type languages are more interesting: precisely the interesting ones are those which are not capable of expressing the relation -- just like the one we are currently investigating.

To make things precise, we show that $\for$ is indeed incapable of expressing the relation. Say that $\for$ expresses Epstein relations iff for any $\r\subseteq\for^2$, any $\varphi,\psi\in\for$, there is $\chi\in\for$, s.t. $\langle\varphi,\psi\rangle\in\r$ iff $\r\vDash\chi$.

Now we can prove the:
\begin{prop}[$\r$ inexpressibility]\label{p: r inexpressibility}
$\for$ does not express Epstein relations.
\end{prop}
\begin{proof}
Let $\r_0=\{\langle\top,\bot\rangle\}$. Assume that there is $\chi\in\for$ such that for any $\r\subseteq\for^2$, $\r\vDash\chi$ iff $\langle\top,\bot\rangle\in\r$. Let $v$ be a valuation. By assumption, $\r_0\vDash\chi$, so also $\langle v,\r_0\rangle\vDash\chi$. It is immediate to observe that $\langle v, \emptyset\rangle\in\SS^{\langle v,\r_0\rangle}$. Thus $\langle v,\emptyset\rangle\vDash\chi$. But $v$ is arbitrary, so $\emptyset\vDash\chi$ which means that $\langle\top,\bot\rangle\in\emptyset$ -- contradiction.
\end{proof}

To conclude our digression, we claim that in the face of the undefinability of our two properties and inexpressibility of the Epstein relation in our language, the task of providing adequate axiomatization for chosen sets of relations becomes a non-trivial endeavour.
\subsection{Logic of symmetric relations $\rF S$}
In order to obtain the logical system $\rF S$, we simply enrich $\rF$ with additional axiom: \[(s)\,\,(p\wto q)\to((q\wto p)\lor\neg(q\to p))\]

\begin{thm} The logic $\rF S$ is sound with respect to the set of symmetric Epstein relations $\r^s$.\end{thm}
\begin{proof}To prove soundness, let $\r\in\R^s$. Let $v$ be arbitrary valuation and let $\m=\model$. Assume $\m\vDash\a\wto\b$. Hence $\a\r\b$. We know that $\r$ is symmetric so also $\b\r\a$. Assume now that $\m\vDash\b\to\a$. This means that $\m\vDash\b\wto\a$.
\end{proof}

\begin{thm}
If $\Sigma\vDash_{\R^s}\a$, then $\Sigma\vdash_{\rF S}\a$.
\end{thm}
\begin{proof}
Assume $\Sigma\nvdash_{\rF S}\a$. Hence $\Sigma\cup\{\neg\a\}$ is $\rF S$-consistent. Let $\splus$ be the $\rF S$-mcs extension of $\Sigma\cup\{\neg\a\}$. Let $\m=\langle\vsplus,\r\rangle$, $\r=\rsplusmin\cup S$ where $S=\{\langle\b,\a\rangle: \a\wto\b\in\splus\}$. To finish the proof, we have to show that 1) $\r$ is symmetric and 2) that $\m$ is the model for $\splus$ i.e. $\m\vDash\splus$. For 1) assume that $\a\r\b$. Hence $\langle\a,\b\rangle\in\rsplusmin$ or $\langle\a,\b\rangle\in S$. If the first disjunct is true, then $\a\wto\b\in\splus$ meaning that $\langle\b,\a\rangle\in S$ and so $\b\r\a$. If it is the case that $\langle\a,\b\rangle\in S$, then $\b\wto\a\in\splus$. This means $\langle\b,\a\rangle\in\rsplusmin$. Notice that in order to prove 2), by fact \eqref{f: canonical set}, it is enough to show that $\r\subseteq\rsplusmax$. Assume, then, that $\langle\a,\b\rangle\in\r$. This means that either 1) $\langle\a,\b\rangle\in\rsplusmin$ or 2) $\langle\a,\b\rangle\in S$. If 1) is the case, then $\a\wto\b\in\splus$, hence also $\langle\a,\b\rangle\in\rsplusmax$. If 2) is the case, then $\b\wto\a\in\splus$. But also by axiom $(s)$ and properties of $\splus$, $\a\wto\b\in\splus$ or $\a\to\b\notin\splus$. This means that $\langle\a,\b\rangle\in\rsplusmax$. This way, we have shown that $\r\subseteq\rsplusmax$. We conclude that $\m\vDash\splus$.
\end{proof}

\subsection{Logic of $\R^n$, the system $\rF N$}
In order to obtain the logical system $\rF N$, we add the following two axioms to the system $\rF$ ($N:=\{n1,\,\,n2\}$): \[(n1)\,\,(\neg p\wto q)\to((p\wto q)\lor\neg(p\to q))\] \[(n2)\,\,((\neg\neg p\wto q)\wedge\neg(\neg p\to q))\to(p\wto q)\]

\begin{thm}[Soundness of $\rF N$]
The logic $\rF N$ is sound with respect to $\R^n$ -- the set of relations fulfilling the (n) condition.
\end{thm}

\begin{proof}
Let $\r\in\R^n$. Let $v$ be arbitrary valuation and $\a,\b$ be arbitrary formulas. Let $\m=\model$. In order to prove validity of $(n1)$ assume that $\m\vDash\neg\a\wto\b$. This means that $\neg\a\r\b$. So also $\a\r\b$. Assume further that $\m\vDash\a\to\b$. This means $\m\vDash\a\wto\b$. For $(n2)$ assume that $\m\vDash(\neg\neg\a\wto\b)\wedge\neg(\neg\a\to\b)$. This means that $\neg\neg\a\r\b$, $\neg\a\r\b$ and also $\a\r\b$. Since $\m\vDash\neg(\neg\a\to\b)$, we know that $\m\nvDash\a$ and $\m\nvDash\b$, which means that $\m\vDash\a\to\b$. We conclude that $\m\vDash\a\wto\b$.
\end{proof}

\begin{thm}[Strong completeness]
The logic $\rF N$ is strongly complete with respect to $\R^n$ i.e. $\Sigma\vDash_{\R^n}\a$ implies $\Sigma\vdash_{\rF N}\a$ for any $\Sigma\cup\{\a\}\subseteq\for$. 
\end{thm}

\begin{proof}
Let $\Sigma\subseteq\for$ and $\a\in\for$. Assume that $\Sigma\nvdash_{\rF N}\a$. This means that $\Sigma\cup\{\neg\a\}$ is $\rF N$-consistent set. Hence it can be extended to maximally $\rF N$-consistent set. Let $\splus$ be such set. Let $\m=\langle\vsplus, \r\rangle$ be such that $\r=\rsplusmin\cup N$, where $N=\{\langle\a,\b\rangle: \neg\a\wto\b\in\splus\}$. To show that $\m\vDash\splus$ we have to show that $\r\subseteq\rsplusmax$. Assume that $\langle\a,\b\rangle\in\r$. If $\langle\a,\b\rangle\in\rsplusmin$, then obviously $\langle\a,\b\rangle\in\rsplusmax$. Assume that $\langle\a,\b\rangle\in N$. This means that $\neg\a\wto\b\in\splus$. By $(n1)$ and maximality of $\splus$, we obtain $\a\wto\b\in\splus$ or $\a\to\b\notin\splus$. Hence by definition \eqref{df S set} we know that $\langle\a,\b\rangle\in\rsplusmax$. This way, we have proven $\r\subseteq\rsplusmax$, which means that $\m\vDash\splus$. Now, we have to show that $\r\in\r^n$ which means that $\neg\a\r\b$ implies $\a\r\b$. Assume that $\neg\a\r\b$. Hence $\langle\neg\a,\b\rangle\in\rsplusmin$ or $\langle\neg\a,\b\rangle\in N$. In the first case $\neg\a\wto\b\in\splus$. So $\langle\a,\b\rangle\in N$, which means that $\a\r\b$. If the second disjunct is true, then $\neg\neg\a\to\b\in\splus$. Either $\neg\a\wto\b\in\splus$, or $\neg\a\wto\b\notin\splus$. If the first disjunct is true, then $\langle\a,\b\rangle\in N$ which finishes the proof. Assume then, that $\neg\a\wto\b\notin\splus$. By axiom $(n1)$, maximality and Modus Ponens we obtain $\neg\a\wto\b\in\splus$ or $\neg(\neg\a\to\b)\in\splus$. Hence $\neg(\neg\a\to\b)\in\splus$. By $(n2)$ we obtain $\a\wto\b\in\splus$. This means that $\langle\a,\b\rangle\in\rsplusmin$, so $\a\r\b$.
\end{proof}
\subsection{The $\rF SN$ system -- the logic of $\R^{sn}$}
By the set of relations $\R^{sn}$, we mean the $\R^s\cap\R^n$.

In order to obtain the logic $\rF SN$ we combine the axioms of $\rF S$ together with those of $\rF N$ and add the following two schemas: 
\[(sn)\,\,(\neg( p\to q)\wedge(\neg p\wto q))\to(q\wto p)\]
\[(ns)\,\,(\neg(\neg p\to q)\wedge(q\wto\neg p))\to(p\wto q)\]

\begin{thm}[Soundness of $\rF SN$]
$\rF SN$ is sound with respect to $\R^{sn}$.
\end{thm}
\begin{proof}
Let $\r\subseteq\R^{sn}$, let $v$ be arbitrary valuation and let $\m=\model$. Assume that $\m\vDash\neg(\a\to\b)\wedge(\neg\a\wto\b)$. Hence $\m\vDash\a$ and $\m\nvDash\b$, which implies $\m\nvDash\b$ or $\m\vDash\a$. We also know that $\neg\a\r\b$. Since $\r\in\r^n$, we know that $\a\r\b$. Also $\b\r\a$ since $\r\in\r^s$. This way, we obtain $\m\vDash\b\wto\a$. For $(ns)$ assume $\m\vDash\neg(\neg\a\to\b)\wedge(\b\wto\neg\a)$. This way, we get $\m\nvDash\a$ or $\m\vDash\b$ and $\b\r\neg\a$. Hence $\neg\a\r\b$ by symmetry and $\a\r\b$ by the (n) condition. We conclude that $\m\vdash\a\wto\b$.
\end{proof}

\begin{thm}[Strong completeness of $\rF SN$]
$\Sigma\vDash_{\R^{sn}}\a$ implies $\Sigma\vdash_{\rF SN}\a$, for any $\Sigma\subseteq\for$, any $\a\in\for$.
\end{thm}
\begin{proof}
Let $\Sigma\cup\{\a\}\subseteq\for$. Assume that $\Sigma\nvdash_{\rF SN}\a$. This means $\Sigma\cup\{\neg\a\}$ is $\rF SN$-consistent. Let $\splus$ be its maximally $\rF SN$-consistent extension. Let $\m=\langle\vsplus,\r\rangle$, $\r=\rsplusmin\cup S\cup N$, where $S$ and $N$ are as above. We have already shown that $S,N\subseteq\rsplusmax$ so the only thing left is to prove that $\r\in\R^{sn}$. First to show that $\r$ is symmetric assume that $\a\r\b$. If $\langle\a,\b\rangle\in\rsplusmin$, then $\langle\b,\a\rangle\in S$. If $\langle\a,\b\rangle\in S$, then $\langle\b,\a\rangle\in\rsplusmin$. Assume then that $\langle\a,\b\rangle\in N$. Hence $\neg\a\wto\b\in\splus$. By axiom $(n1)$, this means that $\a\wto\b\in\splus$ or $\a\to\b\notin\splus$. If $\a\wto\b\in\splus$, then $\langle\b,\a\rangle\in S$. Let us then consider the case where $\a\to\b\notin\splus$. By axiom $(sn)$ we get $\b\wto\a\in\splus$ which means that $\langle\b,\a\rangle\in\rsplusmin$. 

To show that $\neg\a\r\b$ implies $\a\r\b$ assume that $\neg\a\r\b$. If $\langle\neg\a,\b\rangle\in\rsplusmin$, then  $\langle\a,\b\rangle\in N$ by axiom $(n1)$. If $\langle\neg\a,\b\rangle\in N$, then $\langle\a,\b\rangle\in\rsplusmin\cup N$ by previous result for $\rF N$. Assume that $\langle\neg\a,\b\rangle\in S$. This means that $\b\wto\neg\a\in\splus$. Either $\neg\a\wto\b\in\splus$, or $\neg\a\to\b\notin\splus$ by axiom $(s)$. If the former is the case, then $\langle\a,\b\rangle\in N$. Hence assume the latter. By axiom $(ns)$ we obtain $\a\wto\b\in\splus$, which means that $\langle\a,\b\rangle\in\rsplusmin$.
\end{proof}

\section{Interpolation}
We close our study of Epstein semantics with the proof of the interpolation theorem for $\mathcal{F}$. 
Let us start with a formulation of the theorem.
\begin{thm}[Interpolation]\label{t:interpolation}
If $\varphi\to\psi\in\mathcal{F}$ then there is $\chi$ such that $\V(\chi)\subseteq\V(\varphi)\cap\V(\psi)$ and $\varphi\to\chi\in\mathcal{F}$,  $\chi\to\psi\in\mathcal{F}$.
\end{thm}
We shall refer to formulas like $\chi$ in the above formulated interpolation theorem as interpolants (of $\varphi$ and $\psi$). Before moving directly to the proof, we shall first introduce some useful definitions.
\begin{definition}[Realisable pairs]\label{d:realisable pairs}
Let $\Gamma,\Sigma\subseteq\for$. The pair $\langle\Gamma,\Sigma\rangle$ will be called realisable iff there is $\m\in\M$ such that $\m\vDash\varphi$ for any $\varphi\in\Gamma$ and $\m\nvDash\psi$ for any $\psi\in\Sigma$.
\end{definition}
It is easy to observe that by \ref{d:realisable pairs}, $\varphi\to\psi\notin\mathcal{F}$ is equivalent to $\langle\{\varphi\},\{\psi\}\rangle$ being realisable. We shall also introduce the notion of a separable pair.
\begin{definition}[Separable pairs]\label{d:separable pairs}
We say that a pair $\langle\Gamma,\Sigma\rangle$ is separable iff there is $\chi$ such that $\V(\chi)\subseteq\V(\Gamma)\cap\V(\Sigma)$ and both $\langle\Gamma,\{\chi\}\rangle$ and $\langle\{\chi\},\Sigma\rangle$ are not realisable. In such cases, we say that $\{\chi\}$ separates $\langle\Gamma,\Sigma\rangle$.
\end{definition}
Now we can start the proof of theorem \ref{t:interpolation}.
\begin{proof}
Assume that $\varphi\to\psi$ does not have an interpolant. We will prove that the pair $t_0=\langle\Gamma_0,\Sigma_0\rangle=\langle\{\varphi\},\{\psi\}\rangle$ is realisable. By our assumption, we know that $t_0$ is not separable. We shall construct a finite sequence of pairs starting from $t_0$. First let $\varphi_1,...,\varphi_j$ and $\psi_1,...,\psi_k$ be an enumeration of all $\varphi$'s proper subformulas and all $\psi$'s proper subformulas respectively. Let $i<j$. We shall construct the first part of our sequence in the following way:

\begin{align*}
t_{i+1} &=
\begin{cases}
\langle\Gamma_i\cup\{\varphi_{i+1}\},\Sigma_0\rangle, & \text{if }\langle\Gamma_i\cup\{\varphi_{i+1}\},\Sigma_0\rangle \text{ is not separable};\\
\langle\Gamma_i\cup\{\neg\varphi_{i+1}\},\Sigma_0\rangle, & \text{if }\langle\Gamma_i\cup\{\varphi_{i+1}\},\Sigma_0\rangle \text{ is separable}
\end{cases}
\end{align*}
First observe that each $t_i$, $i\leq j$ is not separable. We already mentioned that $t_0$ is not separable (otherwise it would have an interpolant). Assume that $t_{i+1}$ is separable. Then by the construction, both $\langle\Gamma_i\cup\{\varphi_{i+1}\},\Sigma_0\rangle$ and $\langle\Gamma_i\cup\{\neg\varphi_{i+1}\},\Sigma_0\rangle$ are separable. This means that there are $\chi_1,\chi_2$, such that $\V(\chi_1)\subseteq\V(\Gamma\cup\{\varphi_{i+1}\})\cap\V(\Sigma_0)$, $\V(\chi_2)\subseteq\V(\Gamma\cup\{\neg\varphi_{i+1}\})\cap\V(\Sigma_0)$ and all the pairs: $\langle\Gamma_i\cup\{\varphi_{i+1}\},\{\chi_1\}\rangle$, $\langle\{\chi_1\},\Sigma_0\rangle$, $\langle\Gamma_i\cup\{\neg\varphi_{i+1}\},\{\chi_2\}\rangle$ and $\langle\{\chi_2\},\Sigma_0\rangle$ are not realisable. Note that for any $i\leq j$, $\V(\Gamma_i)=\V(\Gamma_0)$, since at each step of the construction we add only $\varphi$'s subformulas. For this reason, $\V(\chi_1\lor\chi_2)\subseteq\V(\Gamma_i)\cap\V(\Sigma_0)$. Both $\langle\Gamma_i,\{\chi_1\lor\chi_2\}\rangle$ and $\langle\{\chi_1\lor\chi_2\},\Sigma_0\rangle$ are not realisable, because otherwise we obtain contradiction with the assumption that $t_{i+1}$ is separable. Hence $t_i$ is separable, which proves the first observation by contraposition.

Now we move on to the second part of our sequence where $m<k$:
\begin{align*}
t_{j+m+1} &=
\begin{cases}
\langle\Gamma_j,\Sigma_m\cup\{\psi_m\}\rangle, & \text{if }\langle\Gamma_j,\Sigma_m\cup\{\psi_m\}\rangle \text{ is not separable};\\
\langle\Gamma_j,\Sigma_m\cup\{\neg\psi_m\}\rangle, & \text{if }\langle\Gamma_j,\Sigma_m\cup\{\psi_m\}\rangle \text{ is separable}
\end{cases}
\end{align*}
Analogously, we will prove that for any $m\leq k$, $t_{j+m}$ is not separable. Let $t_{j+m+1}$ be separable. Hence there are $\chi_1,\chi_2$, such that $\V(\chi_1)\subseteq\V(\Gamma_j)\cap\V(\Sigma_m\cup\{\psi_{m+1}\})$, $\V(\chi_2)\subseteq\V(\Gamma_j\})\cap\V(\Sigma_m\cup\{\neg\psi_{m+1}\})$ and all the pairs: $\langle\Gamma_j,\{\chi_1\}\rangle$, $\langle\{\chi_1\},\Sigma_m\cup\{\psi_{m+1}\}\rangle$, $\langle\Gamma_j,\{\chi_2\}\rangle$ and $\langle\{\chi_2\},\Sigma_m\cup\{\neg\psi_{m+1}\}\rangle$ are not realisable. In this case $\{\chi_1\wedge\chi_2\}$ separates $\langle\Gamma_j,\Sigma_m\rangle$ which gives us the desired result.

Now, we are going to prove that $t_{j+k}=\langle\Gamma_j,\Sigma_k\rangle$ is realisable. We will define the model $\m=\model$ in the following way: 
\begin{align*}
v(p)=1 \text{ iff }& p\in\Gamma_j \text{ or } \neg p\in\Sigma_k\\
\langle\varphi,\psi\rangle\in\r \text{ iff }&\text{at least one of the followin is true:}\\
&\text{either }\varphi\wto\psi\in\Gamma_j,\tag{a}\label{a}\\
&\text{or }\varphi\wi\psi\in\Gamma_j,\tag{b}\label{b}\\
&\text{or }\neg(\varphi\wto\psi)\in\Sigma_k,\tag{c}\label{c}\\
&\text{or }\neg(\varphi\wi\psi)\in\Sigma_k,\tag{d}\label{d}
\end{align*}
We will prove that:
\begin{enumerate}
\item for any $\chi\in\Gamma_j$, $\m\vDash\chi$ iff $\chi\in\Gamma_j$,
\item for any $\sigma\in\Sigma_k$, $\m\nvDash\sigma$ iff $\sigma\in\Sigma_k$
\end{enumerate}

We shall start with 1.
Assume that $\chi=p\in\V$. If $p\in\Gamma_j$, then by definition $\m\vDash p$. Let $\m\vDash p$. $p\in\Gamma_j$ gives us the result immediately, so assume $\neg p\in\Sigma_k$. By our construction either $p\in\Gamma_j$ or $\neg p\in\Gamma_j$. If $\neg p\in\Gamma_j$, then $\{\neg p\}$ separates $\langle\Gamma_j,\Sigma_k\rangle$ which gives us contradiction, hence $p\in\Gamma_j$. Let $\chi=\gamma\wedge\delta$. Assume that $\gamma\wedge\delta\in\Gamma_j$. Hence $\gamma\in\Gamma_j$ or $\neg\gamma\in\Gamma_j$ and similarly $\delta\in\Gamma_j$ or $\neg\delta\in\Gamma_j$. If $\neg\gamma\in\Gamma_j$, then it is easy to see that $\{\bot\}$ separates $\langle\Gamma_j,\Sigma_k\rangle$, which leads to contradiction. The case when $\neg\delta\in\Sigma_k$ is analogous. Hence both $\gamma\in\Gamma_j$ and $\delta\in\Gamma_j$, so by inductive hypothesis $\m\vDash\gamma$ and $\m\vDash\delta$, which means $\m\vDash\gamma\wedge\delta$.
Assume now that $\m\vDash\gamma\wedge\delta$. Obviously $\m\vDash\gamma$ and $\m\vDash\delta$, so by inductive hypothesis $\gamma\in\Gamma_j$, $\delta\in\Gamma_j$. If $\neg(\gamma\wedge\delta)\in\Gamma_j$, then it is easy to see that $\{\bot\}$ separates $\langle\Gamma_j,\Sigma_k\rangle$, so $\gamma\wedge\delta\in\Gamma_j$.
The rest of the boolean cases are treated similarly, so we move on to Epstein connectives. Let $\chi=\gamma\wto\delta$. If $\m\vDash\gamma\wto\delta$, then $\langle\gamma,\delta\rangle\in\r$ which means that one of the four cases from the definition is true.
The first case \ref{a} immediately leads to the desired result. Assume \ref{b}. 
Assume for reductio that $\gamma\wto\delta\notin\Gamma_j$, which means $\neg(\gamma\wto\delta)\in\Gamma_j$. Then $\{\bot\}$ separates $\langle\Gamma_j,\Sigma_k\rangle$, because there is no $\m'\in\M$, such that $\m'\vDash\gamma\wi\delta$ and $\m'\vDash\neg(\gamma\wto\delta)$. Contradiction, so it must be the case that $\gamma\wto\delta\in\Gamma_j$. Similarly, assume \ref{c} and $\neg(\gamma\wto\delta)\in\Gamma_j$. then it can easily be seen that $\{\neg(\gamma\wto\delta)\}$ separates $\langle\Gamma_j,\Sigma_k\rangle$, so $\gamma\wto\delta\in\Gamma_j$. Assume \ref{d} holds and $\neg(\gamma\wto\delta)\in\Gamma_j$. Then it can be easily observed that $\{\neg(\gamma\wto\delta)\}$, as well as $\{\neg(\gamma\wi\delta)\}$ separates $\langle\Gamma_j,\Sigma_k\rangle$, so $\gamma\wto\delta\in\Gamma_j$.

For the other direction assume $\gamma\wto\delta\in\Gamma_j$. Hence we already know that $\langle\gamma,\delta\rangle\in\Gamma_j$. 
If both $\gamma\in\Gamma_j$ and $\delta\notin\Gamma_j$, then it is easy to see that $\{\bot\}$ separates $\langle\Gamma_j,\Sigma_k\rangle$. For this reason it is the case that $\gamma\notin\Gamma_j$ or $\delta\in\Gamma_j$. By inductive hypothesis, we get $\m\vDash\gamma$ and $\m\vDash\delta$. This immediately leads to $\m\vDash\gamma\wto\delta$. 

Assume now that $\chi=\gamma\wi\delta$. If $\m\vDash\gamma\wi\delta$, then $\langle\gamma,\delta\rangle\in\r$, so again we have four cases to consider. We can say already that \ref{b} gives us the result immediately. Also we have $\m\vDash\gamma$ and $\m\vDash\delta$. By hypothesis: $\gamma\in\Gamma_j$ and $\delta\in\Gamma_j$. Assume \ref{a}. If $\gamma\wi\delta\notin\Gamma_j$, then we obtain the contradiction, since $\{\bot\}$ separates $\langle\Gamma_j,\Sigma_k\rangle$. Now assume \ref{c} or \ref{d}. If $\gamma\wi\delta\notin\Gamma_j$, then in both cases $\{\neg(\gamma\wi\delta)\}$ separates $\langle\Gamma_j,\Sigma_k\rangle$, so we obtain $\gamma\wi\delta\in\Gamma_j$. For the other direction assume $\gamma\wi\delta\in\Gamma_j$. Hence $\langle\gamma,\delta\rangle\in\r$. It is also easy to observe that it must be the case that $\gamma\in\Gamma_j$ and $\delta\in\Gamma_j$ (otherwise we get the separation by $\{\bot\}$), so by inductive hypothesis we get $\m\vDash\gamma\wi\delta$.


For 2. assume that $\m\nvDash p$. Hence by definition of $\m$, we have $\neg p\notin\Sigma_k$, which means $p \in\Sigma_k$. Assume $p\in\Sigma_k$. If $p\in\Gamma_j$, then $\{p\}$ separates $\langle\Gamma_j,\Sigma_k\rangle$, so $p\notin\Gamma_j$. Also $\neg p\notin\Sigma_k$, so $\m\nvDash p$. We shall skip the boolean cases and move to Epstein connectives right away. Let $\chi=\gamma\wto\delta$. Assume that $\m\nvDash\gamma\wto\delta$. Hence we have two possibilities: either i)$\langle\gamma,\delta\rangle\notin\r$ or ii)$\m\vDash\gamma$ and $\m\nvDash\delta$. By i) we get that neither of the conditions \ref{a}--\ref{d} holds, so \ref{c} gives us $\neg(\gamma\wto\delta)\notin\Sigma_k$ which means $\gamma\wto\delta\in\Sigma_k$. Assume ii) holds. Hence by inductive hypothesis $\neg\gamma\in\Sigma_k$ and $\delta\in\Sigma_k$. If $\gamma\wto\delta\notin\Sigma_k$, then $\neg(\gamma\wto\delta)\in\Sigma_k$. But in such case it is easy to see that $\{\neg\bot\}$ separates $\langle\Gamma_j,\Sigma_k\rangle$, because there is no $\m'$ such that $\m'\vDash\gamma\wto\delta$ and $\m'\vDash\gamma$, $\m'\nvDash\delta$. For this reason we obtain $\gamma\wto\delta\in\Sigma_k$.

For the other direction assume $\gamma\wto\delta\in\Sigma_k$. Assume now that \ref{a} holds. Then obviously $\{\gamma\wto\delta\}$ separates $\langle\Gamma_j,\Sigma_k\rangle$. So \ref{a} does not hold. For similar reasons, it must be the case that \ref{b} does not hold. It is obvious that \ref{c} cannot be true. If \ref{d} holds, then it also must be the case that $\{\neg\bot\}$ separates $\langle\Gamma_j,\Sigma_k\rangle$, since there is no $\m'\in\M$ such that $\m'\vDash\gamma\wi\delta$ and $\m'\nvDash\gamma\wto\delta$. Hence \ref{d} is not true, so $\langle\gamma,\delta\rangle\notin\r$ which means $\m\nvDash\gamma\wto\delta$.

Let $\chi=\gamma\wi\delta$. Assume $\m\nvDash\gamma\wi\delta$. Hence either i)$\langle\gamma,\delta\rangle\notin\r$, or  ii)$\m\nvDash\gamma$ or $\m\nvDash\delta$. Again, i) gives us the result immediately, so assume ii). By inductive hypothesis we get $\gamma\in\Sigma_k$ or $\delta\in\Sigma_k$. Assume for reductio that $\neg(\gamma\wi\delta)\in\Sigma_k$. There is no model $\m'$ such that $\m\vDash\gamma\wi\delta$ and ($\m\nvDash\gamma$ or $\m\nvDash\delta$), so $\{\neg\bot\}$ separates $\langle\Gamma_j,\Sigma_k\rangle$ -- contradiction. Hence $\gamma\wi\delta\in\Sigma_k$.
For the other direction assume $\gamma\wi\delta\in\Sigma_k$. Obviously \ref{d} does not hold. If \ref{c} holds, then it can be easily checked that it must be the case that $\gamma\in\Sigma_k$ or $\delta\in\Sigma_k$ (otherwise we get separation by $\{\neg\bot\}$). Hence $\m\nvDash\gamma$ or $\m\nvDash\delta$ which by inductive hypothesis leads to $\m\nvDash\gamma\wi\delta$. When \ref{b} holds, then $\{\gamma\wi\delta\}$ separates $\langle\Gamma_j,\Sigma_k\rangle$, so \ref{b} cannot be true. Assume \ref{a}. Assume further that $\neg\gamma\in\Sigma_k$ and $\neg\delta\in\Sigma_k$. But then $\{\gamma\wto\delta\}$ separates $\langle\Gamma_j,\Sigma_k\rangle$(to see that observe that for no $\m'$: $\m'\vDash\gamma\wto\delta$, $\m'\vDash\gamma,\m'\vDash\delta$ and $\m'\nvDash\gamma\wi\delta$). It must be the case then that $\gamma\in\Sigma_k$ or $\delta\in\Sigma_k$ which by inductive hypothesis means either $\m\nvDash\gamma$ or $\m\nvDash\delta$, so $\m\nvDash\gamma\wi\delta$. 

This way, we have proven that $t_{j+m+1}$ is realisable which means that $\m\vDash\varphi$ and $\m\nvDash\psi$. This ends the proof.
\end{proof}

We finish with two additional remarks.
\begin{remark}
Observe that the relation in model $\m$ from the proof of theorem \ref{t:interpolation} could be defined in a following way: $\langle\varphi,\psi\rangle\in\r$ iff $\neg(\varphi\wto\psi)\notin\Gamma_j$ or $\neg(\varphi\wi\psi)\notin\Sigma_k$ or $\varphi\wto\psi\notin\Gamma_j$ or $\varphi\wi\psi\notin\Sigma_k$. Note that, although the change seem to be purely cosmetic, thus defined relation can be radically different from the one defined in the previous part! To see that, let $\langle\Gamma_j,\Sigma_k\rangle=\langle\{p\},\{q\}\rangle$. According to the first definition $\r=\emptyset$, whilst the current one gives us $\r=\for\times\for$. Actually, the same could be done with the definition of $v$.
\end{remark}

One may also wonder if the analogous result of interpolation theorem could be obtained for Epstein implication.
\begin{remark}
The interpolation theorem holds trivially for Epstein implication, that is $\varphi\wto\psi\in\mathcal{F}$ implies existence of $\chi$, such that $\V(\chi)\subseteq\V(\varphi)\cap\V(\psi)$, $\varphi\wto\chi\in\mathcal{F}$ and $\chi\wto\psi\in\mathcal{F}$.
\end{remark}
\begin{proof}
Let $\r=\emptyset$ and $v$ be arbitrary. Observe that for any $\varphi,\psi\in\for$, we have $\model\nvDash\varphi\wto\psi$. For this reason $\varphi\wto\psi\notin\mathcal{F}$.
\end{proof}

\section{Conclusion}
In the paper, we have used model-theoretic (mainly the $\SS$--set) techniques to obtain some results on generalised Epstein semantics. We hope that our work is a step towards building a serious theory of Epstein semantics. We showed that each Epstein model has uncountably many equivalent models. We have also proven that the logic of Epstein structures is the $\SS$-set invariant fragment of $\cpl$. Additionally, we isolated continuum many extensions of $\mathcal{F}$ which are Epstein-incomplete.
We have presented a strategy for proving completeness with respect to undefinable sets of Epstein relations. What is important, is that we did not have to incorporate any extra rules to the system $\rF$ to achieve the result -- we have simply enriched $\rF$ with extra axioms. What also seems worth mentioning, is that providing the axiom system for the intersection of $\R^s$ and $\R^n$, was not just a matter of combining the system $\rF S$ and $\rF N$. To obtain completeness, we had to include additional axioms. Both undefinability and completeness were proven by means of the $\SS$-set construction. 
\bibliographystyle{plain}
\bibliography{krawczyk_epstein_semantics}
Krzysztof Aleksander Krawczyk
\\Department of Logic
\\Jagiellonian University in Krakow, Poland
\\krzysztof1.krawczyk@uj.edu.pl

\end{document}